\documentclass{amsart}

\RequirePackage{color}

\usepackage{amsmath,amsthm,amssymb,mathrsfs,amscd}
\usepackage{mathtools}
\usepackage[all]{xy}
\usepackage[english]{babel}
\usepackage[utf8]{inputenc}
\usepackage[colorlinks=false, linktocpage=true]{hyperref}
\usepackage{xcolor}

\newtheoremstyle{thmstyle}{}{}{\itshape}{}{\bfseries}{ }{5pt}{}
\newtheoremstyle{defstyle}{}{}{}{}{\bfseries}{ }{5pt}{}
\newtheoremstyle{remstyle}{}{}{}{}{\bfseries}{ }{5pt}{}

\theoremstyle{thmstyle}
\newtheorem{thm}{Theorem}[section]

\newtheorem{lem}[thm]{Lemma}

\newtheorem{prop}[thm]{Proposition}

\newtheorem{cor}[thm]{Corollary}

\theoremstyle{defstyle}

\newtheorem{example}[thm]{Example}
\newtheorem{question}[thm]{Question}

\theoremstyle{remstyle}

\newtheorem{remark}[thm]{Remark}

\newcommand{\Hom}{\operatorname{Hom}}
\newcommand{\Fac}{\operatorname{Fac}}
\newcommand{\Ext}{\operatorname{Ext}}

\DeclareMathOperator{\End}{End}

\DeclareMathOperator*{\modu}{mod}
\DeclareMathOperator*{\proj}{proj}
\DeclareMathOperator*{\add}{add}

\DeclareMathOperator*{\Filt}{Filt}

\DeclareMathOperator*{\ind}{ind}
\DeclareMathOperator*{\brick}{brick}
\DeclareMathOperator*{\simp}{simp}

\DeclareMathOperator*{\wide}{wide}
\DeclareMathOperator*{\sbrick}{sbrick}

\DeclareMathOperator*{\pd}{pd}
\DeclareMathOperator*{\GL}{GL}
\DeclareMathOperator*{\rep}{rep}
\DeclareMathOperator*{\Irr}{Irr}

\newcommand{\T}{\mathcal{T}}
\newcommand{\W}{\mathcal{W}}

\newcommand{\Du}{\mathrm{D}}
\newcommand{\trig}{\tau\text{-rigid}}

\begin{document}

\title{Brick infinite algebras admit infinitely many non-$\tau$-rigid bricks}

\author[Kaveh Mousavand, Charles Paquette]{Kaveh Mousavand, Charles Paquette}

\address{Kaveh Mousavand: Representation Theory and Algebraic Combinatorics Unit, Okinawa Institute of Science and Technology (OIST), Japan}
\email{mousavand.kaveh@gmail.com}
\address{Charles Paquette: Department of Mathematics and Computer Science, Royal Military College of Canada, Kingston ON, Canada}
\email{charles.paquette.math@gmail.com}


\begin{abstract}
For finite dimensional algebras over algebraically closed fields, we settle a question previously known only for certain families of algebras. More specifically, motivated by some foundational interactions between bricks and $\tau$-rigid modules, we show that a given algebra is brick infinite if and only if it admits infinitely many bricks which are not $\tau$-rigid. This proves the $\tau$-analogue of an open conjecture asserting that if (almost) all bricks over an algebra $A$ are rigid, then $A$ should be brick-finite. In retrospect, we strengthen some recent contributions to the study of a series of challenging open problems related to the $2$nd brick-Brauer-Thrall conjecture. 
Moreover, motivated by some of our arguments, we pose the question whether there exists any algebra that admits an infinite semibrick consisting of Ext-orthogonal rigid bricks. In connection with this, we present an algebra $A$ of rank $n$ that admits a semibrick of cardinality strictly greater than $n$ consisting of Ext-orthogonal rigid bricks.
\end{abstract}

\maketitle

\section{Introduction}

Throughout, $k$ denotes an algebraically closed field and $A$ a finite dimensional, basic and connected, associative unital $k$-algebra of rank $n$. By Gabriel's theorem, we may henceforth assume $A \cong kQ/I$, where $Q$ is a finite quiver and $I \subseteq kQ$ is an admissible ideal. That being the case, we also identify modules in $\modu A$ (the category of finite dimensional left $A$-modules) with
representations of the bound quiver $(Q,I)$. As usual, $\ind(A)$ denotes the set of (isomorphism classes of)
indecomposables in $\modu A$, and $A$ is called \emph{representation-finite} if $\ind(A)$ is finite.

\medskip

A module $X \in \modu A$ is a \emph{brick} (or a \emph{Schur representation}) if $\End_A(X) \cong k$, and $\brick(A)$ denotes the set of isomorphism classes of bricks. Then, the algebra $A$ is \emph{brick-finite} if $\brick(A)$ is a finite set. Bricks are known to be central objects in the modern representation theory of algebras. For instance, they govern wide subcategories and torsion pairs, label the Hasse quiver of torsion classes, arise as the stable objects of King's stability conditions, and, as shown in the more recent studies, they also control the $\tau$-tilting finiteness of
$A$; for more details on bricks and their recent developments, see \cite{MP4} and the long list of references therein.

\medskip

To state our main problem and results, recall that $M \in \modu A$ is said to be \emph{$\tau$-rigid} provided $\Hom_A(M, \tau M)=0$, where $\tau$ is the Auslander--Reiten translate. By a fundamental result of \cite{DIJ}, $A$ is brick-finite if and only if it is $\tau$-tilting finite, that is, if and only if it admits only finitely many indecomposable $\tau$-rigid modules. Moreover, the authors gave an explicit injective \emph{brick--$\tau$-rigid correspondence} from the set $i\,\trig(A)$ of indecomposable $\tau$-rigid modules into $\brick(A)$. It is further known that if $A$ is brick-finite, the aforementioned brick--$\tau$-rigid correspondence is in fact a bijection; notably, the converse remains unknown in general!

\medskip

Among many similar studies conducted in recent years, we have investigated the relationship between bricks and $\tau$-rigid modules; see \cite{MP1,MP2,MP3} and the references therein.
In particular, in \cite[Theorem~4.3]{MP1}, we showed that every brick is $\tau$-rigid if and only if $A$ is \emph{locally representation-directed}, where the latter class of algebras was introduced by Dr\"axler \cite{Dr} and gives a generalization of the representation-directed algebras from a different point of view. 
It is notable that any locally representation-directed algebra is necessarily representation-finite. 
Thus, it is natural to relax the hypothesis from \emph{all} bricks to \emph{almost all} bricks. The following is previously recorded in \cite{MP2} (see also \cite[Proposition~8.7]{MP4}), where it is already affirmatively established for the class of $E$-tame algebras (therefore, all tame algebras).

\begin{question}\label{q:main}
If all but finitely many bricks of $A$ are $\tau$-rigid, is $A$ necessarily brick-finite?
\end{question}

This problem can be seen as a $\tau$-rigid analogue of the (still open) question of whether the rigidity of almost all bricks forces brick-finiteness (see Remark \ref{Rem: on rigid-brick conj}). 
Furthermore, Question \ref{q:main} is closely tied to an important problem, nowadays called the ``second brick-Brauer-Thrall conjecture" on distribution of bricks, as well as to a question of Demonet \cite{De}, concerned with the completeness of the $g$-vector fan (also called $\tau$-tilting fan) in $K_0(\proj A)_\mathbb{Q}$. For more details on these challenging problems, see \cite{MP2,MP4} and references therein. The purpose of this note is to settle Question~\ref{q:main}, \emph{affirmatively and in full generality}, without any tameness hypothesis. More specifically, our main result is the following.

\begin{thm}\label{thm:main-intro}
Suppose that $A$ admits infinitely many $\tau$-rigid bricks. Then $A$ admits infinitely many
bricks that are not $\tau$-rigid. 
\end{thm}

Equivalently, the above theorem asserts that an algebra cannot have an infinite family of $\tau$-rigid bricks but only finitely many non-$\tau$-rigid bricks. Because every brick-finite algebra trivially has all but finitely many bricks $\tau$-rigid, observe that Theorem~\ref{thm:main-intro} yields an affirmative answer to Question~\ref{q:main}. 
We also state our main result in the more common language of rigidity, which requires no prior knowledge nor explicit use of Auslander-Reiten theory and the AR translate $\tau$. In that regard, for $X$ in $\modu A$, let $\Fac(X)$ denote the set of modules obtained as a quotient of some powers of $X$. 
More explicitly,
$$\Fac(X):=\{M\in \modu A \,|\, X^{\oplus d} \twoheadrightarrow M, \text{ for some } d \in \mathbb{Z}_{>0} \}.$$
Then, define $\Ext^1_A(X,\Fac(X))=0$ if $\Ext^1_A(X,M)=0$, for all $M \in \Fac(X)$.

\begin{cor}\label{cor:main-intro}
The algebra $A$ is brick-finite if and only if all but finitely many of its bricks are $\tau$-rigid.
Consequently, for an algebra $A$, the following are equivalent:
\begin{enumerate}
    \item $A$ is brick-infinite;
    \item $\Hom_A(X,\tau X)\neq 0$, for infinitely many $X \in \brick(A)$;
    \item $\Ext^1_A(X,\Fac(X))\neq 0$, for infinitely many $X \in \brick(A)$.
\end{enumerate}

\end{cor}

It is instructive to compare Corollary~\ref{cor:main-intro} with the corresponding statement for \emph{all} bricks. Passing from ``all'' to ``almost all'' bricks, the conclusion necessarily weakens
from a class of representation-finite algebras to brick-finite algebras. We remark that
brick-finite algebras need not be representation-finite. In fact, there exist representation-infinite algebras (of tame or wild type) which satisfy the assumption of the preceding corollary, hence they are brick-finite. Among the abundance of such algebras, preprojective algebras of Dynkin quivers have received a lot of attention; for instance, see \cite{Mi} and the references therein.
\medskip

The initial step of our proof is the reduction to minimal brick-infinite algebras, introduced in \cite{Mo} (also see \cite{Wa}) and further systematically studied in \cite{MP3} through the lens of $\tau$-tilting theory. Then, we develop the argument according to the representation type of algebras under consideration.
More specifically, the proof strategy combines three main ingredients of independent interest, outlined below.
First, thanks to the
$\tau$-convergence property introduced in \cite{MP1}, over a minimal brick-infinite algebra we have that an infinite family of $\tau$-rigid bricks induces an integral $g$-vector lying outside the $\tau$-tilting fan. 
Second, by a recent theorem of Asai \cite{As2}, such a non-rigid integral $g$-vector yields an \emph{infinite} semibrick. Third, when the algebra is not strictly wild, we show that any infinite semibrick contains an infinite sub-semibrick whose members are pairwise extension-orthogonal; over a minimal brick-infinite algebra this yields arbitrarily large $\tau$-rigid modules, which is absurd. The strictly wild case is handled separately and directly; it is known that strictly wild algebras always possess infinitely many bricks of a fixed dimension, and only finitely many bricks of a given dimension can be $\tau$-rigid.

\medskip

The paper is organized as follows. Section~\ref{sec:prelim} collects the necessary background on
bricks, $\tau$-rigidity, semibricks and wide subcategories, representation varieties, the
$\tau$-tilting fan, and the numerical torsion classes. Section~\ref{sec:proofs}
contains the three results outlined above and the proof of Theorem~\ref{thm:main-intro} and Corollary~\ref{cor:main-intro},  where we also compare our results with some important open problems.

\section{Preliminaries and background}\label{sec:prelim}

In this section, we briefly recall the notions and known results used in the rest. For all the undefined terminology and rudimentary facts from the representation theory of finite dimensional algebras, we refer to \cite{ASS}.

\subsection{Bricks and \texorpdfstring{$\tau$}{tau}-rigid modules}\label{Subsection: bricks and tau-rigids}
For $M \in \modu A$, let $|M|$ denote the number of pairwise non-isomorphic indecomposable
summands of $M$; in particular $|A|=n$. Recall that $M$ is said to be \emph{rigid} if $\Ext^1_A(M,M)=0$, and is $\tau$-\emph{rigid} if $\Hom_A(M,\tau M)=0$. A $\tau$-rigid module $M$ with $|M|=n$ is called $\tau$-\emph{tilting}. We use the following standard facts; for details, see
\cite{ASS,AIR}; and also \cite[\S2]{MP4}.

\begin{itemize}
\item[(R1)] Thanks to the Auslander--Reiten duality, there exists a natural surjection
$\Hom_A(M,\tau M) \twoheadrightarrow \overline{\Hom}_A(M,\tau M) \cong \Du\Ext^1_A(M,M)$, where
$\overline{\Hom}$ denotes morphisms modulo those factoring through an injective, and $\Du(-):=\Hom_k(-,k)$.
Particularly, $\tau$-rigid modules are rigid, but not the converse.
\item[(R2)] If $\pd_A M \leq 1$, then $\Hom_A(M,\tau M)=\overline{\Hom}_A(M,\tau M)$, so that one has
$\Hom_A(M,\tau M)\cong \Du\Ext^1_A(M,M)$. Hence a rigid module of projective dimension at most one
is $\tau$-rigid \cite[Remark~2.2(c)]{MP4}.
\item[(R3)] Every $\tau$-rigid module $M$ satisfies $|M|\leq n$. This bound is due to Skowro\'nski;
see \cite[VIII, Lemma~5.3]{ASS}, as well as \cite[\S2]{AIR}.
\item[(R4)] If $M$ is $\tau_A$-rigid and $J\subseteq A$ is a two-sided ideal with $JM=0$, then
$M$ is $\tau_{A/J}$-rigid. This classical fact, going back to Skowro\'nski, can be found in
\cite{ASS}; see also \cite[Remark~2.2(b)]{MP4}.
\end{itemize}

\subsection{Semibricks and wide subcategories}
A \emph{semibrick} is a (possibly infinite) set $\mathcal{S}\subseteq \brick(A)$ of pairwise
Hom-orthogonal bricks, that is, for all distinct pairs $X,Y\in \mathcal{S}$, we have $\Hom_A(X,Y)=0=\Hom_A(Y,X)$.
We write $\sbrick(A)$ for the set of all semibricks in $\modu A$. A full subcategory $\W\subseteq\modu A$ is a \emph{wide subcategory} if it is closed under kernels, cokernels and extensions; thus $\W$ is an exact abelian subcategory. We let $\wide(A)$ denote the set of wide subcategories. For a class $\mathcal{C}$ of modules, $\Filt(\mathcal{C})$ is the full subcategory of modules admitting a finite filtration with subquotients in $\mathcal{C}$.

The following classical correspondence, due to Ringel \cite{Ri}, is well known and will be useful in our study below (see also \cite{As3,MS, MP4}).

\begin{prop}\label{prop:semibrick-wide}
The assignment $\mathcal{S}\mapsto \Filt(\mathcal{S})$ is a bijection from $\sbrick(A)$ to
$\wide(A)$, with inverse $\W\mapsto \simp(\W)$, where $\simp(\W)$ is the set of simple objects of
the abelian category $\W$. In particular, for $\mathcal{S}\in\sbrick(A)$ the simple objects of
$\Filt(\mathcal{S})$ are precisely the elements of $\mathcal{S}$.
\end{prop}

Since a wide subcategory is closed under extensions, for $X,Y\in\W$ the Yoneda group of
$1$-extensions computed in $\W$ coincides with $\Ext^1_A(X,Y)$. We will use this identification
without further comment.

\subsection{Representation varieties and orbits}
For each dimension vector $d$, let $\rep(A,d)$ be the affine variety parametrizing the $A$-modules of
dimension vector $d$, on which the group $\GL(d)$ acts by conjugation with orbits $\mathcal{O}_X$, where $X\in\rep(A,d)$. The variety $\rep(A,d)$ has only finitely many irreducible components, and $\Irr(A,d)$ denotes the set of all these components. We further recall the \emph{Voigt's Lemma}: if $\Ext^1_A(X,X)=0$, then $\mathcal{O}_X$ is open in the component of $\rep(A,d)$ containing it; see
\cite[\S2]{MP2}. Since an irreducible variety contains at most one dense orbit and $\rep(A,d)$ has
only finitely many components, a rigid module of dimension vector $d$ has an open orbit in its
component, and there are only finitely many isoclasses of such modules. Combined with our earlier remark \textup{(R1)} from Subsection \ref{Subsection: bricks and tau-rigids}, we record this well-known observation for later use.

\begin{lem}\label{lem:rigid-finite}
For each dimension vector $d$, the algebra $A$ admits only finitely many isoclasses of rigid
modules of dimension vector $d$. In particular, only finitely many bricks of a fixed dimension vector can be $\tau$-rigid.
\end{lem}

\subsection{The \texorpdfstring{$\tau$}{tau}-tilting fan and rigid \texorpdfstring{$g$}{g}-vectors}
We identify $K_0(\proj A)\cong\mathbb{Z}^n$ via the basis given by the indecomposable projectives. Furthermore, we set $K_0(\proj A)_{\mathbb{R}}:=K_0(\proj A)\otimes_{\mathbb{Z}}\mathbb{R}$. Elements of
$K_0(\proj A)$ are often called \emph{$g$-vectors}. By \cite{AIR,DIJ}, the $g$-vectors of indecomposable
$\tau$-rigid pairs are the rays of a simplicial fan $F_A$ in $K_0(\proj A)_{\mathbb{R}}$, to which, we refer to as the \emph{$\tau$-tilting fan} of $A$ (also known as the $g$-fan of $A$). A $g$-vector (real, rational or integral) is \emph{rigid} if it
lies in the support of $F_A$, namely, in a cone of $F_A$; rigidity is constant along rays. By \cite[Theorem~4.7]{As2} (see also \cite{DIJ}), $A$ is brick-infinite if and only if $F_A$ is not
complete in $K_0(\proj A)_{\mathbb{R}}=\mathbb{R}^n$, equivalently, if and only if there exists a
non-rigid \emph{real} $g$-vector $\theta\in K_0(\proj A)_{\mathbb{R}}$. It is not known in general
whether brick-infiniteness forces the existence of a non-rigid \emph{integral} (equivalently,
rational) $g$-vector: First posed as a question by Demonet \cite[Question 3.49]{De}, and now treated as a conjecture, he asked if every rational element of $K_0(\proj A)$ being rigid implies that $A$ is brick-finite (for more details, see \cite{MP2} and \cite{Pf}). Producing an integral non-rigid $g$-vector in the situation of interest to
us is precisely the role of Proposition~\ref{prop:nonrigid-g} below.

\subsection{Numerical torsion classes and non-rigid $g$-vectors}
We first recall that $\T\subseteq\modu A$ is a \emph{torsion class} if it is closed under quotients and
extensions, and that $\T(\mathcal{C})$ denotes the smallest torsion class containing a class
$\mathcal{C}$. Following King and \cite[\S2]{As2}, for $\theta\in K_0(\proj A)_{\mathbb{R}}$, viewed
as a linear form on $K_0(\modu A)_{\mathbb{R}}$, the \emph{numerical torsion class}
\[
\overline{\T}^{\theta} \ :=\ \{\, M\in\modu A \ \mid \ \theta(N)\geq 0 \text{ for every quotient } N \text{ of } M\,\}
\]
is functorially finite precisely when $\theta$ is rigid. An important ingredient of our argument is the following recent theorem of Asai, which provides infinite semibricks attached to non-rigid integral
$g$-vectors. It rests on the theory of bicompact torsion classes and on the geometric invariant theory of quiver representations treated by King \cite{Ki}.

\begin{prop}[{\cite[Theorem~5.4]{As2}}]\label{prop:asai}
Assume $k$ is algebraically closed and let $\theta\in K_0(\proj A)$ be an integral $g$-vector. If
$\theta$ is not rigid, then there exists an infinite semibrick $\mathcal{S}$ such that
$\overline{\T}^{\theta}=\T(\mathcal{S})$. In particular, $A$ admits an infinite semibrick.
\end{prop}

\subsection{Wild and strictly wild algebras}
Let $K_3$ denote the $3$-Kronecker quiver and let $H = k(K_3)$ the corresponding path algebra. Recall that $A$ is \emph{strictly wild} if there exists a $k$-linear exact functor
$F\colon \modu H \to\modu A$ which is full and faithful and preserves indecomposability. One can take $F$ as $F = - \otimes_H M_A$ where $M$ is an $H-A$ bimodule such that $M$ is finitely generated projective as a left $H$-module, see \cite[Chapter XIX, Cor. 1.6]{SS}. In particular, note that $F$ sends non-isomorphic bricks to non-isomorphic bricks.  
It is also worth emphasizing that if $M,N$ have the same dimension in $\modu H$, then $F(M), F(N)$ have the same dimension in $\modu A$.

\medskip

In the following, we will freely use two standard facts. First, it is well known that the path algebra of a (connected) wild quiver is strictly wild. Second, if some wide subcategory $\W\subseteq\modu A$ contains a full additive subcategory
that is equivalent to $\rep(Q)$ for a wild quiver $Q$, then $A$ is strictly wild. Indeed, composing
$\modu H \to\modu kQ\cong\rep(Q)\hookrightarrow\W\hookrightarrow\modu A$ yields a $k$-linear fully faithful exact embedding $\modu H \to \modu A$.

\section{Almost all bricks being \texorpdfstring{$\tau$}{tau}-rigid implies brick-finiteness}\label{sec:proofs}

In our treatment of the main result, we first discuss the strictly wild case, which should be known to experts.

\begin{prop}\label{prop:wild}
If $A$ is strictly wild, then $A$ admits infinitely many bricks of the same dimension vector. In
particular, $A$ admits infinitely many bricks which are not $\tau$-rigid.
\end{prop}

\begin{proof}
It is well known that the path algebra $H=kK_3$ of the $3$-Kronecker quiver admits infinitely many
bricks of the same dimension vector (for instance, consider the indecomposable representations of dimension
vector $(1,1)$). We consider a $k$-linear functor $F = - \otimes_H M \colon\modu H\to\modu A$ considered above, where $M$ is an $H-A$ bimodule which is finitely generated projective over $H$. It takes non-isomorphic bricks to non-isomorphic bricks, and takes modules of equal dimension to modules of equal
dimension. Applying $F$ to the above infinite family yields infinitely many non-isomorphic bricks of $A$,
all of the same dimension. By Lemma~\ref{lem:rigid-finite}, only finitely many of them can
be $\tau$-rigid, so infinitely many are not.
\end{proof}

The following result is an interesting feature of semibricks over algebras that are not strictly wild. It essentially says that over every non-strictly wild algebra, an infinite semibrick always contains an infinite sub-semibrick consisting of pairwise $\Ext_A^1$-orthogonal bricks. 

\begin{prop}
    
\label{lem:thin}
Assume that $A$ is not strictly wild, and let $\mathcal{S}=\{B_i\}_{i\in\mathbb{N}}$ be an infinite
semibrick. Then there exists an infinite sub-semibrick $\{C_i\}_{i\in\mathbb{N}}\subseteq\mathcal{S}$
such that $\Ext^1_A(C_i,C_j)=0$ for all $i\neq j$.
\end{prop}

\begin{proof}
For $i\in\mathbb{N}$ put
\[
E(i\!\to)=\{\,j\neq i \mid \Ext^1_A(B_i,B_j)\neq 0\,\}, \qquad
E(\to\! i)=\{\,j\neq i \mid \Ext^1_A(B_j,B_i)\neq 0\,\}.
\]
We claim that $|E(i\!\to)|\leq 4$ and $|E(\to\! i)|\leq 4$ for every $i$.

Suppose for contradiction that $E(i\!\to)$ contains five distinct indices $j_1,\dots,j_5$, and set
$\mathcal{S}'=\{B_i,B_{j_1},\dots,B_{j_5}\}$ and $\W'=\Filt(\mathcal{S}')$. By
Proposition~\ref{prop:semibrick-wide}, $\W'$ is a wide subcategory whose simple objects are exactly
the six bricks of $\mathcal{S}'$. Consider the full subcategory $\mathcal{D}\subseteq\W'$ of objects
$M$ admitting a short exact sequence
\[
0\longrightarrow M'\longrightarrow M\longrightarrow M''\longrightarrow 0
\]
with $M'\in\add\{B_{j_1},\dots,B_{j_5}\}$ and $M''\in\add B_i$; equivalently, the objects of $\W'$ of loewy length at most two
whose socle lies in $\add\{B_{j_1},\dots,B_{j_5}\}$ and whose top lies in $\add B_i$. As
$\mathcal{S}'$ is a semibrick and the $B$'s are bricks over a basic algebra over an algebraically closed field, we have
$\End_A(B_i)\cong k\cong\End_A(B_{j_l})$ and $\Hom_A(B_i,B_{j_l})=0=\Hom_A(B_{j_l},B_i)$ for all $l$.
As $\Hom_A(B_{j_l},B_i)=0$, for $M,N\in\mathcal{D}$, every morphism $M\to N$ maps the submodule layer of $M$ (that is, the largest submodule of $M$ in $\add\{B_{j_1},\dots,B_{j_5}\}$) into the
submodule layer of $N$. Therefore every morphism $M\to N$ is a morphism of the corresponding
two-step filtrations. 

Let $Q'$ be the quiver with a source vertex $v$, five sink vertices $w_1,\dots,w_5$, and
$\dim_k\Ext^1_A(B_i,B_{j_l})\geq 1$ arrows $v\to w_l$ for $1\leq l\leq 5$. A representation of $Q'$
with dimension vector having dimension $b$ at $v$ and $a_l$ at $w_l$ is exactly the datum of an extension class in
\[
\Ext^1_A\!\big(B_i^{\,b},\textstyle\bigoplus_{l}B_{j_l}^{\,a_l}\big)
\ \cong\ \bigoplus_{l}\Ext^1_A(B_i,B_{j_l})^{\,a_l\times b},
\]
and the above analysis of morphisms shows that this assignment is an an equivalence of $\rep(Q')$ onto
$\mathcal{D}$. Now the
underlying graph of $Q'$ contains the star with five outgoing arrows from a source vertex, which is a connected wild quiver. By the discussion at the end of
Section~\ref{sec:prelim}, the embedding $\rep(Q')\simeq\mathcal{D}\subseteq\W'\subseteq\modu A$ makes
$A$ strictly wild, contradicting our hypothesis. Therefore $|E(i\!\to)|\leq 4$. The bound
$|E(\to\! i)|\leq 4$ follows by the dual argument.

Now form the (undirected) graph $G$ on the vertex set $\mathbb{N}$ in which $i$ and $j$ are joined
whenever $\Ext^1_A(B_i,B_j)\neq 0$ or $\Ext^1_A(B_j,B_i)\neq 0$. By the claim, each vertex of $G$ has
degree at most $|E(i\!\to)|+|E(\to\! i)|\leq 8$. An infinite graph of bounded degree contains an
infinite independent set: choose $i_1$, delete it together with its at most $8$ neighbours, choose
$i_2$ among the infinitely many remaining vertices, and iterate. The resulting family
$\{C_t\}_{t\in\mathbb{N}}=\{B_{i_t}\}_{t\in\mathbb{N}}$ is pairwise non-adjacent in $G$, that is,
$\Ext^1_A(C_s,C_t)=0=\Ext^1_A(C_t,C_s)$ for all $s\neq t$. This is the desired infinite sub-semibrick.
\end{proof}

To emphasize the significance of the previous proposition, in the following we make a remark on strictly wild algebras that must be known to experts.

\begin{remark}
Note that $A$ is strictly wild if and only if there exists an infinite semibrick $\{B_i\}_{i \in \mathbb{N}}$ such that $\Ext^1_A(B_i, B_j)\ne 0$ for all $i,j$. Indeed, if $A$ is a strictly wild algebra, then using the strictly wild hereditary algebra $H=kK_3$ and dimension vector $(1,1)$, we have an infinite semibrick $\{B_i\}_{i\in \mathbb{N}}$ over $H$ with the property that $\Ext^1_{H}(B_i, B_j)\ne 0$ for all $i,j$. The corresponding representation embedding $F:\modu H \to \modu A$ yields the infinite semibrick $\{F(B_i)\}_{i \in \mathbb{N}}$ over $A$ with $\Ext_A^1(F(B_i), F(B_j))\ne 0$ for all $i,j$. This is because, although $F$ does not preserve Ext-spaces in general, it does preserve non-split short exact sequences. For the converse, one just uses ideas similar to the proof of Proposition \ref{lem:thin}.
\end{remark}

\medskip

We are now ready to prove our main result. Let us recall that an algebra $A$ is said to be \emph{minimal brick-infinite} if
$\brick(A)$ is infinite while $\brick(A/J)$ is finite for every nonzero two-sided ideal $J$. In the following, we use two important properties of minimal brick-infinite algebras established in our previous work \cite{MP1,MP3}.

\begin{prop}[{\cite{MP3}, see also \cite[Proposition~6.1]{MP1}}]\label{prop:mintau}
Let $A$ be a minimal brick-infinite algebra. Then, all but finitely many indecomposable $\tau$-rigid
$A$-modules are faithful and of projective dimension one.
\end{prop}

\begin{prop}[{\cite[Theorem~6.8]{MP1}}]\label{prop:nonrigid-g}
If $A$ is minimal brick-infinite and admits infinitely many $\tau$-rigid bricks, then there exists
a non-rigid $g$-vector.
\end{prop}

\begin{proof}[Sketch of proof]
By Proposition~\ref{prop:mintau}, all but finitely many of the infinitely many $\tau$-rigid bricks
have projective dimension one. As shown in \cite[\S6]{MP1}, existence of an infinite family of $\tau$-rigid bricks of projective dimension one forces the \emph{$\tau$-convergence property}, and
\cite[Theorem~6.8]{MP1} yields a rational $g$-vector outside the $\tau$-tilting fan. Rescaling
along its ray gives an integral non-rigid $g$-vector.
\end{proof}

\begin{thm}\label{thm:main}
Suppose that $A$ admits infinitely many $\tau$-rigid bricks. Then $A$ admits infinitely many bricks
that are not $\tau$-rigid.
\end{thm}

\begin{proof}
Assume, for contradiction, that $A$ has only finitely many non-$\tau$-rigid bricks. Together with
the hypothesis, this says that $A$ is brick-infinite and that \emph{almost all} bricks of $A$ are
$\tau$-rigid.

We first argue how to reduce the problem to minimal brick-infinite algebras, then we treat two cases separately. 
Since $A$ is finite dimensional and brick-infinite, we may choose a two-sided ideal $J$ such that $B:=A/J$ is minimal brick-infinite. 
Observe that, from our assumptions, it follows that almost all bricks in $\modu B$ are $\tau_B$-rigid as well. Indeed, $\modu B$ is a full subcategory of $\modu A$, so $\brick(B)\subseteq\brick(A)$. Also, it follows from (R4) that if $X\in\brick(B)$ is not $\tau_B$-rigid then $X$ is not $\tau_A$-rigid.
Therefore, the non-$\tau_B$-rigid bricks of $B$ form a subset of the non-$\tau_A$-rigid bricks of $A$, which is finite. As $B$ is brick-infinite, it therefore admits infinitely many $\tau_B$-rigid
bricks. We now derive a contradiction, distinguishing two cases according to whether or not $B$ is strictly wild.

\medskip

If $B$ is strictly wild, then by Proposition~\ref{prop:wild} the algebra $B$ admits infinitely many
bricks that are not $\tau_B$-rigid. This contradicts the above fact that almost all
bricks of $B$ are $\tau_B$-rigid.

\medskip

Finally, we may assume that $B$ is not strictly wild. Since $B$ is minimal brick-infinite with infinitely many
$\tau_B$-rigid bricks, Proposition~\ref{prop:nonrigid-g} provides a non-rigid integral $g$-vector
$\theta\in K_0(\proj B)$. As $k$ is algebraically closed, Proposition~\ref{prop:asai} then yields an
infinite semibrick $\mathcal{S}=\{B_i\}_{i\in\mathbb{N}}$ of $B$.

Because almost all bricks of $B$ are $\tau_B$-rigid, all but finitely many members of $\mathcal{S}$
are $\tau_B$-rigid; discarding the finitely many exceptions, we assume that every $B_i$ is a
$\tau_B$-rigid brick. By Proposition~\ref{prop:mintau}, all but finitely many indecomposable
$\tau_B$-rigid $B$-modules have projective dimension one; discarding finitely many more members of
$\mathcal{S}$, we may assume that each $B_i$ is moreover of projective dimension at most one.

Applying Proposition~\ref{lem:thin} to the not-strictly-wild algebra $B$ and the infinite semibrick
$\{B_i\}$ yields an infinite sub-semibrick $\{C_i\}_{i\in\mathbb{N}}$ with
$\Ext^1_B(C_i,C_j)=0$ for $i\neq j$.

Fix any $m> {\rm rk}(B)$ and set $M:=C_1\oplus\cdots\oplus C_m$. Because $C_i$ has projective dimension $\leq 1$ with $0=\Ext^1_B(C_i, C_j) \cong D\Hom_B(C_j, \tau C_i)$ for each $j$, then $M$ is $\tau_B$-rigid. But the $C_i$ are pairwise non-isomorphic, so $|M|=m> {\rm rk}(B)$,
contradicting the bound (R3) for $\tau_B$-rigid modules over $B$, whose rank is at most ${\rm rk}(B)$.

\smallskip
Note that in either case we reached a contradiction. This completes the proof.
\end{proof}

\begin{cor}\label{cor:main}
The algebra $A$ is brick-finite if and only if all but finitely many of its bricks are $\tau$-rigid.
\end{cor}

\begin{proof}
If $A$ is brick-finite, then $\brick(A)$ is finite, so the condition holds vacuously. Conversely,
suppose all but finitely many bricks of $A$ are $\tau$-rigid but, for contradiction, that $A$ is
brick-infinite. Then $A$ admits infinitely many $\tau$-rigid bricks, so by Theorem~\ref{thm:main} it
also admits infinitely many bricks that are not $\tau$-rigid, contradicting the hypothesis. Hence $A$
is brick-finite.
\end{proof}

\begin{remark}
Corollary~\ref{cor:main} answers Question~\ref{q:main} with no tameness assumption, extending the $E$-tame case of \cite[Proposition~8.7]{MP4}.
We emphasize that our proof uses the $\tau$-rigidity hypothesis, which is known to be stronger than rigidity. In fact, it is expected that a similar implication holds under the weaker rigidity assumption; see Remark \ref{Rem: on rigid-brick conj}. We also note that in the language of the brick--$\tau$-rigid correspondence \cite{DIJ}, Corollary~\ref{cor:main} says that $A$ is brick-finite as soon as this correspondence acts like identity for a co-finite subset of its codomain. 
These new results strengthen some previous results from \cite[Theorem 1.1]{MP1} on the behavior of $\tau$-rigid bricks. 
That being the case, it leads to further progress towards some challenging conjectures on the study of bricks, including those summarized in \cite{MP4}.
\end{remark}

\medskip

We end with supplementary remarks to compare our main question with the similar question that uses the weaker rigidity assumption. We also relate these questions with some fundamental and closely related open problems. 
The next remark also completes the proof of Corollary \ref{cor:main-intro}, more specifically, the equivalence of $(1)-(3)$.

\begin{remark}\label{Rem: on rigid-brick conj}
An important open problem that we call the ``rigid-brick conjecture" states: \emph{If $\Ext^1_A(X,X)=0$ for all but possibly finitely many $X\in \brick(A)$, is $A$ necessarily brick-finite?} An affirmative answer to this question is expected; yet unknown in full generality. That is to say, for every brick-infinite algebra $A$, it is conjectured that we always have $\Ext^1_A(X,X)\neq 0$, for infinitely many $X\in \brick(A)$. 
This itself is an immediate consequence of \cite[Conjecture 1.3(2)]{Mo}, nowadays often known as the ``Second brick-Brauer-Thrall (2nd bBT) conjecture"; a similar statement also appeared in \cite[Conjecture 2]{STV}. For more details, recent developments, and related problems, see \cite{MP2,MP4} and the references therein.

\smallskip

Observe that, from \cite[Proposition 5.8]{AS}, it follows that $\Hom_A(X,\tau X)=0$ if and only if $\Ext^1_A(X,\Fac(X))=0$. 
Thus, Question \ref{q:main} can be reformulated as follows: \emph{If $\Ext^1_A(X,\Fac(X))=0$ for all but possibly finitely many $X\in \brick(A)$, is $A$ necessarily brick-finite?} This articulation better highlights the conceptual link to the above-mentioned open ``rigid-brick conjecture", because it does not use the Auslander-Reiten translate $\tau$ and entirely relies on the more elementary condition $\Ext^1_A(X,M)=0$, for all $M \in \Fac(X)$. 
Therefore, our Theorem \ref{thm:main-intro} settles this (equivalent) question, and Corollary \ref{cor:main-intro} asserts that for every brick-infinite algebra $A$, we always have $\Ext^1_A(X,\Fac(X))\neq 0$, for infinitely many $X\in \brick(A)$. 
\end{remark}

\section{Size of semibricks consisting of Ext-orthogonal rigid bricks}

In the proof of Theorem \ref{thm:main}, we ended up with a semibrick of size larger than the rank of the algebra, consisting of pairwise Ext-orthogonal $\tau$-rigid bricks. This led to a contradiction because under our assumptions, those bricks could be assumed to have projective dimension at most one, which implied the existence of a $\tau$-rigid module with too many non-isomorphic indecomposable summands. 

\medskip

This leads to the following natural question about rigidity of semibricks: is it possible to have a semibrick consisting of Ext-orthogonal rigid bricks whose size exceeds the rank of the algebra? We note that any such semibrick has an associated wide subcategory that is equivalent to the module category of a (finite dimensional) semisimple algebra. 

\medskip

It is well known that a wide subcategory $\mathcal{W}$ of $\modu A$ is functorially finite if and only if $\mathcal{W}$ is equivalent to $\modu B$, for a finite dimensional algebra $B$; for instance, see \cite[Proposition 4.12]{En}. Therefore, one can naturally ask whether, in this case, ${\rm rk}(B) \le {\rm rk}(A)$. In fact, this inequality holds if $\mathcal{W}$ is left-finite, namely, if the associated torsion class of $\mathcal{W}$ is functorially finite; see \cite[Corollary 2.10]{As1}. Left finite wide subcategories are known to form a subset of the functorially finite wide subcategories.
The example below shows that functorially finiteness alone is not enough to guarantee that ${\rm rk}(B) \le {\rm rk}(A)$.
In fact, we present a family of algebras of rank $3$ each of which admitting semibricks consisting of pairwise Ext-orthogonal rigid bricks of size strictly larger than $3$.

\begin{example} \label{example: sizes of rigid semibricks}
    Fix an integer $r \ge 4$ and let $Q_r$ be the quiver with three vertices $1,2,3$ and arrows
\[
  \alpha_1,\dots,\alpha_r \colon 1 \longrightarrow 2,
  \qquad
  \beta_1,\dots,\beta_r \colon 2 \longrightarrow 3.
\]
Define
\[
  A_r := kQ_r \big/ \big\langle\, \beta_j \alpha_i \;:\; i \neq j \,\big\rangle.
\]
 For $1 \le \ell \le r$, let $X_\ell$ be the representation of $A_r$
having vector space $k$ at each vertex and
where only the arrows $\alpha_\ell, \beta_\ell$ act non-trivially. One can easily check that the set $\{X_\ell \mid 1 \le \ell \le r\}$ is a semibrick consisting of $\Ext$-orthogonal rigid bricks. We note that the wide subcategory $\mathcal{W}(\{X_\ell \mid 1 \le \ell \le r\})$ is equivalent to $\modu k^r$, hence is functorially finite whose rank exceeds the rank of $A_r$. 
\end{example}

The above argument and example lead to the following question.

\begin{question}
    Let $A$ be a finite dimensional associative $k$-algebra. Is it true that any semibrick consisting of pairwise Ext-orthogonal rigid bricks is necessarily finite? If so, can we find an upper bound on the sizes of such finite semibricks?
\end{question}

\textbf{Acknowledgments.}
KM was supported by Early-Career Scientist JSPS Kakenhi grant number 24K16908. CP was supported by the Natural Sciences and Engineering Research Council of Canada (RGPIN-2026-05988) and by the Canadian Defence Academy Research Programme. We thank Sota Asai and Osamu Iyama for helpful discussions related to the problem treated in this manuscript. 
Part of this work was conducted while KM was on a research visit in Québec, facilitated by the Tomlinson Fellowship Visiting Scholar Funding from Bishop’s University. He is particularly grateful to Prof. Thomas Br\"ustle for his hospitality.

\bigskip

\textbf{Use of Artificial Intelligence.} The familly of examples $A_r$ in Example \ref{example: sizes of rigid semibricks} was found during an exploratory search with Claude (by Anthropic).

\end{document}